\documentclass[12pt]{amsart}
\usepackage{amssymb,amsmath,amsfonts,amsthm,amscd,enumerate}

\usepackage{graphics}
\usepackage[colorlinks=true]{hyperref}
\usepackage{amssymb,amsmath,amsfonts,amsthm,amscd,mathrsfs}

\input xy
\xyoption{all}
\newcommand{\CC}{\mathbb C}
\newcommand{\RR}{\mathbb R}

\newcommand{\maD}{\mathcal{D}}

\newcommand{\ZZ}{\mathbb{Z}}
\newcommand{\Sch}{\mathscr{S}}
\newcommand{\NN}{\mathbb{N}}

\newcommand{\maM}{\mathcal{M}}

\newcommand{\maG}{\mathcal{G}}

\newcommand{\maU}{\mathcal{U}}
\newcommand{\maL}{\mathcal{L}}

\newcommand{\smooth}{\mathcal{C}^{\infty}}

\newcommand{\Op}[1]{\operatorname{#1}}
\newcommand{\ep}{\varepsilon}
\newcommand{\smoothing}{\Psi^{\!-\!\infty}}

\newcommand{\ip}[1]{\langle #1 \rangle}

\newcommand{\gp}[1]{\Gamma^{\infty}(#1)}
\newcommand{\heis}{\underset{H_n}{*}}
\newcommand{\compsupp}{\mathcal{E}'}

\theoremstyle{definition}
\newtheorem{definition}{Definition}[section]

\theoremstyle{plain}
\newtheorem{theo}[definition]{Theorem}
\newtheorem{prop}[definition]{Proposition}
\newtheorem{lem}[definition]{Lemma}
\newtheorem{cor}[definition]{Corollary}

\theoremstyle{remark}
\newtheorem{remark}[definition]{Remark}
\newtheorem{eg}[definition]{Example}
\begin{document}

\textcolor{red}{
 \title[]{Rapidly converging approximations and regularity theory}}
\author{Shantanu Dave}
\address{University of Vienna, Faculty of Mathematics, Nordbergstrasse 15, 1090 Vienna, Austria}
\email{shantanu.dave@univie.ac.at}
\thanks{Supported by FWF grant Y237-N13 of the Austrian Science Fund.}

\maketitle
\begin{abstract}
We consider distributions on a closed compact manifold $M$ as  maps on smoothing operators. Thus spaces of maps between $\smoothing(M)\rightarrow  \smooth(M)$  are   considered as generalized functions. For any collection of regularizing processes we produce an algebra of generalized functions  and a diffeomorphism  equivariant embedding of distributions into this  algebra. We  provide examples   invariant under  certain group actions. The regularity  for such generalized functions is provided in terms of a certain tameness of maps between graded Frech\'et spaces. This  notion of regularity  implies the  regularity in Colombeau algebras in the $\maG^{\infty}$ sense. 
\end{abstract}

\section{Introduction}
Regularization of nonsmooth structures such as  distributions and  discontinuous metrics  by smooth approximates has been an important ingredient of many problems in mathematics and physics. The choice of regularizing process is often dictated by their interaction with different operators involved  and their symmetries.  In numerical processes one desires that the regularizations converge optimally.

By a regularizing approximation we mean a net of smoothing operator $T_{\ep}$ such that $T_{\ep}u$ is a smooth function  and $\lim_{\ep\rightarrow 0}T_{\ep}u=u$.
By fixing  the asymptotic properties of a regularizing approximation one can study the regularity and singularity of the  nonsmooth objects 
 in terms of asymptotic behaviour of the  approximation.

 Here we shall consider regularizing approximations that are in a sense optimal in view of Lorant Schwartz's theorem that states the impossibility of constructing an associative product on distributions consistent with continuous functions.
We shall call such approximations  \emph{rapidly converging approximations}.

Let $M$ be a compact  Riemannian manifold. The Weyl's asymptotic formula provides an asymptotic estimate on the growth of eigenvalues of the  associated Laplace operator $\Delta$ on $M$.  For a Schwartz function $F(x)\in\Sch(\RR)$ let $F_{\ep}(x)=F(\ep x)$. Then  for a suitable choice of Schwartz function $F$ the  net of smoothing operator $T_{\ep}=F_{\ep}(\Delta)$ provides the basic example of  rapidly converging approximations in view of Weyl's theorem. In the special case of compact Lie groups $T_{\ep}$ can be  obtained in terms of convolution  with characters of irreducible representations directly from Peter-Weyl theorem.

 In order to study regularity of a distribution $u\in\maD'(M)$ in terms of the approximation $T_{\ep}u$ one is naturally led to consider  it as a map from smoothing operators $\smoothing(M)$ to $\smooth(M)$, \[\Theta_u:\smoothing(M)\rightarrow \smooth(M)\quad \Theta_u(T):=T(u)\,\forall\,\,T\in\smoothing(M).\]
Here the  Sobolev regularity of $u$ can be interpreted in terms of  degree of tameness  of the map $\Theta_u$ with respect to certain grading on the two Frech\'et spaces $\smooth(M)$ and $\smoothing(M)$. For example a smooth function $f$  provides a map $\Theta_f$  which is tame  for all possible degrees of tameness and this character  classifies all  smooth maps.  This point of view immediately provides us with a natural way to study local regularity of $u$. For instance one can easily identify singular support and the wavefront set of $u$ either from $\Theta_u$ or $T_{\ep}u$.

This leads one to consider a general Frech\'et  smooth  map \[\phi:\smoothing(M)\rightarrow \smooth(M)\] as a generalized  function space for nonlinear  operators. The action of diffeomorphisms and pseudodifferential operators on $\maD'(M)$ can naturally be extended to  the space $E$ of all such maps. The pointwise algebra structure on $E$ obtained from the algebra $\smooth(M)$  provides meaningful subalgebras which correspond to regularity of a generalized function $\phi\in E$. 

 For instance the regularity of such generalized functions  can be measured in two different ways.  Either we can modify the  notion of degree of tameness in which case pseudodifferential operator change the degree of tameness  analogous to their mapping properties on Sobolev spaces.  Or else  we can consider maps with certain asymptotic behaviour for the smooth functions  $\phi(T_{\ep})$ where we choose  the rapidly  converging  approximations $T_{\ep}$ to belong to a fixed set $\maL$ invariant under symmetries of a given problem. 

  The fact that the net of operators $T_{\ep}$ induce a sheaf-morphism  from the sheaf of distributions $\maD'(M)$ to the sheaf $\maG^s(M)$ of special algebra of Colombeau with smooth dependence of parameter plays an important role 
in the present interpretation of  local regularity properties of distributions and more generally for generalized functions in $E$. 

In Euclidian space rapidly converging approximations can be constructed from convolution with a  net of mollifier  converging appropriately to the delta distributions (see \cite{c1}). Another such construction can be carried out  on $\mathbb{R}^{2n+1}$  by convolution by a mollifier $\rho_{\ep}$ on the Heisenberg group. These further yield  rapidly converging approximations on $\RR^n$ by  apply Fourier transform  with respect to the  Schr\"odinger representation to each $\rho_{\ep}$. We note that although each of the above approximations characterize local regularity   of distributions in much the same way, they tend to isolate different properties of global regularity and the growth at infinity.  This is  because they preserve different large scale structure on $\RR^n$. In the present article we shall study only the local behaviour of regularity.

\section{Preliminaries}\label{pre}
Let $X$ be a locally convex (Hausdorff) topological vector space then one can associate
a generalized locally convex space $\maG_X$ (see \cite{Garetto}) as follows.
Let $I$ be the interval $(0,1)$. Define the smooth moderate nets on $X$ to be  smooth maps 
\[I\rightarrow X\quad\ep\rightarrow x_{\ep}\] 

such that for all continuous semi-norms $\rho$ there exists  an integer $N$
such that
\begin{eqnarray}\label{asympt}
|\rho(x_{\epsilon})|\sim O(\epsilon^N)  
\quad \textrm{ as }\, \,\epsilon \rightarrow
 0
\end{eqnarray}
Here as usual  by $f(\ep)\sim  O(g(\ep))$ as $\ep\rightarrow 0$ we mean there exists an  $\ep_0>0$ and a constant $C>0$ such that  $f(\ep)<Cg(\ep)$  for $\ep<\ep_0$.
 We denote the set of all  moderate smooth nets on $X$ by  $E(X)$.
 Similarly we can define the negligible nets to be the smooth maps $x_{\ep }$ such that
for all continuous seminorms $\rho$ and for all $N$  Equation \eqref{asympt} holds. We  shall denote the set  of all smooth negligible nets by $N(X)$.

 The generalized locally convex space over $X $ is then defined to be the quotient,
\[\maG_X:=E(X)/N(X).\]

One notes that in defining $E(X)$ and $N(X) $  it suffices to restrict to a family of
seminorms that generate the locally convex topology on $X$. If  $x_{\ep}$ is a moderate net in $E(X)$ then  the element it represents in the  quotient $\maG_X$  shall be represented by  $\ip{x_{\ep}}$.

When $X=\smooth(M)$ is the space of smooth functions on a manifold $M$ then we
also represent $\maG^s(M):=\maG_{\smooth(M)}$.
Also for $X=\CC$ the space $\maG_{\CC}$ inherits a  ring structure from $\CC$ and we call it  the space of generalized numbers and denote it by $\tilde{\CC}$. Every $\maG_X$ is naturally a $\tilde{\CC}$ module, and hence  is often referred to as the $\tilde{\CC}$ module associated with $X$.  The sharp topology on $\tilde{\CC}$ is the topology generated by sets of the form $U_{x,p}$ where $x\in\tilde{\CC}$ and $p$ is an integer and
\[U_{x,p}:=\{\ip{x_{\epsilon}-\epsilon^p,x_{\epsilon}+\epsilon^p}| \ip{x_{\epsilon}}=x\}.\]
Any continuous seminorm $\rho$ on a locally convex spaces $X$  by definition provides a  map $\tilde{\rho}:E(X)\rightarrow E(\CC)$ by applying $\rho$ to each component. In fact $\tilde{\rho}$ descends to a map from $\maG_X\rightarrow \tilde{\CC}$.
The sharp topology on any $\tilde{\CC}$ module $\maG_X$ shall be defined to be the weakest topology that makes each $\tilde{\rho}$ above continuous.

We recall the functoriality of the above construction \cite{Scarpa},
\begin{lem}\label{lcs}
If $\phi:X\rightarrow Y$ is a continuous linear map between locally
convex spaces $X $ and $Y$ then there is a natural induced map $\phi_*:\maG_X\rightarrow
\maG_Y$ defined on the representatives as
$\phi_*(\ip{x_{\epsilon}})=\ip{\phi(x_\epsilon)}$. Further $\phi_*$ is continuous with respect to sharp topology.
\end{lem}
\begin{proof}
A continuous linear map $\phi$ maps smooth net $x_{\ep}$ to smooth net $\phi_*(x_{\ep})$.
If $\tau$ is a continuous seminorm on $Y$ then $\tau\circ \phi$ is a continuous seminorm on $X$. Thus if $x_{\epsilon}$ satisfies an estimate with $\tau\circ \phi$ in $E(X)$ or $N(X)$ then $\phi(x_{\epsilon})$ satisfies the exact same estimates with respect to $\tau$ in $E(Y)$ or $N(Y)$. Thus $\phi_*$ is well-defined.
 Since basic open set $U$  in $\maG_Y$ are pull-back of open sets in $\tilde{\CC}$ by some seminorm $\rho$, then $\phi_*^{-1}(U)$ is a pullback of an open set  with respect to $\rho\circ \phi$.

\end{proof}
For example any smooth map between two manifold $f:M\rightarrow N$ gives  rise
to a pull back map $f^*:\maG^s(N)\rightarrow \maG^s(M)$. As  a consequence we can define a presheaf of algebras on $M$ by  assigning to any open set $U\subseteq M\rightarrow \maG^s(U)$. The restriction maps are given by the pull back under inclusions, that is $i:U\hookrightarrow V$ is an inclusion of open sets then $i^*:\maG^s(V)\rightarrow \maG^s(U)$ is the restriction map. This presheaf is  in fact a fine sheaf. Thus in particular we can define the support of a global section $x\in\maG^s(M)$  as usual to be the complement of the  biggest open subset of $M$  on which   $x$ restricts to $0$. 

 For any locally convex space $X$ we can also define a subalgebra $\maG^{\infty}_X$ of regular elements of $\maG_X$. These are all elements in $\maG_X$   such that there exists an integer $N$ so that  \eqref{asympt} holds  independent of the   seminorm $\rho$ chosen. Again we shall denote  by $\maG^{\infty}(M)$ the algebra $\maG^{\infty}_{\smooth(M)}$.  The algebra $\maG^{\infty}(M)$ provides  regularity features for analysis of generalized functions  and operation in $\maG^s(M)$ in a  way that $\smooth(M)$ provides  these features in $\maD'(M)$.For instance:
\begin{enumerate}[(a)]
\item  \it{Singular support}: For $\phi\in\maG^s(M)$ the singular support is defined as the  complement of largest open  set $\maU$ on which  the restriction  $\phi_{\maU}$ is in $\maG^{\infty}(M)$.
\item \it{ Wavefront set}: Let $P$ be an order  $0$ classical pseudodifferential operator  and let $\operatorname{char}(P)\subseteq T^*M$ be the characteristic set of $P$, that is  the $0$ set of its principal symbol. Then we can define the generalized wavefront set of a generalized function $\phi$ as:
\[WF_g(\phi):=\cap_{P\phi\in\maG^{\infty}(M)} \operatorname{char}(P)\quad P\in\Psi^0_{cl}(M).\]
\item \it{Hypoellipticity}: We can define an operator $P$ to be $\maG^{\infty}(M)$ hypoelliptic if
\[Pu\in\maG^{\infty}(M)\Longrightarrow u\in\maG^{\infty}(M).\]
\end{enumerate}
 Under appropriate circumstances  the above notions  are extensions of the  same in the distributional sense. Lemma \ref{local}  provides a general example.

\section{Moderate approximate units}\label{regu}

Let $M$ be a closed manifold. Let $\Omega$ be the bundle of $1$-densities on $M$.  By distributions on $M$ we mean the duel space $\maD'(M)=\smooth(M:\Omega)'$. A continuous linear operator $\maD'(M)\rightarrow \smooth(M)$  is called a smoothing operator. The space of all smoothing operators  shall be denoted  by $\Psi^{\!-\!\infty}(M)$ and its forms an ideal in the algebra of pseudodifferential operators $\Psi^{\infty}(M)$.  By identifying an operator to its kernel, smoothing operators can be viewed as a Frech\'et space of smooth sections of a vector bundle namely,
\[\smoothing(M)=\gp{M\times M:\pi_2^*\Omega},\]
 where $\pi_2:M\times M\rightarrow M$ is the projection on the second component.

Now we introduce certain nets of smoothing operators  that shall play the role of delta nets generated from a mollifier. These are the regularizing processes we are interested in.

\begin{definition}\label{mau}
A net of smoothing operators $T_{\ep}$ is called a moderate approximate unit if:
\begin{enumerate}[(a)]
\item  Its a moderate net that is $T_{\ep}\in E_{\smoothing(M)}$. That is $T_{\ep}$ satisfies  \eqref{asympt} with respect to any seminorm on $\smoothing(M)$.
\item  For any $u\in \maD'(M)$ 
\[\underset{\ep\rightarrow 0}{lim}\, T_{\ep}u=u.\]
\item For any smooth function $f\in\smooth(M)$ the approximation $T_{\ep}f\rightarrow f$ converges rapidly in the sense that given a  seminorm $\rho$ on smooth functions
\[\rho(T_{\ep}f-f)\sim O(\ep^N) \, \textrm{ for all}\, N\in\ZZ.\]
\end{enumerate}
\end{definition}
 The set of all Moderate approximate units shall be denoted by $\maU(M)$ or sometimes with $\maU$ for simplicity. The set $\maU$ is  is closed under the obvious action of the diffeomorphism group.
\begin{lem}
Let $\chi$ be a  diffeomorphism. Let $T_{\ep}\in\maU$ be a moderate  approximate unit then the $\ep$-wise push forword $\chi_*T_{\ep}$ is also a moderate approximate unit. 
\end{lem}
\begin{proof}
Since the push foreword map $\chi_*:\smoothing(M)\rightarrow \smoothing(M)$ is defined as 
\[\chi_*(T)(f):=\chi^*(T(\chi^{*^{-1}}f)),\] where $\chi^*$ is a pull back of functions. It is  clear that  $\chi_*$ is  continuous linear map, by Lemma \ref{lcs} it maps  moderate nets to moderate nets. In  particular $\chi_*(T_{\ep})$ is moderate net in $E_{\smoothing(M)}$.

For the same reason $\chi_*$  maps all  negligible nets in $\smooth(M)$  to negligible nets.  First  we observe that
\[\chi_*T(f)-f=\chi^*(T\chi^{{-1}^*}f)-f
=\chi^*(T\chi^{{-1}^*}f-\chi^{{-1}^*}f)\]
Therefore
\[T_{\ep}\chi^{{-1}^*}f-\chi^{{-1}^*}f\in N_{\smooth(M)}\Rightarrow\chi^*(T_{\ep}\chi^{{-1}^*}f-\chi^{{-1}^*}f)\in  N_{\smooth(M)}.\]
By continuity of $\chi$
\begin{align*}
\underset{\ep\rightarrow 0}{lim}\, \chi_*T_{\ep}u&=\underset{\ep\rightarrow 0}{lim}\, \chi^*(T_{\ep}\chi^{{-1}^*}u)\\
&=\chi^*(\underset{\ep\rightarrow 0}{lim}\, T_{\ep}\chi^{{-1}^*}u)=u
\end{align*}
Thus  $\chi_*T_{\ep}$ satisfies Definition \ref{mau}.
\end{proof}
The following proposition follows directly form the definition of moderate approximate units and underlines one of the reasons for the definition.
\begin{prop}
A moderate approximate unit $T_{\ep}$ provides an embedding of  the distributions $\maD'(M)$ into the space of  smooth  special algebra $\maG^s(M)$  by $u\rightarrow  T_{\ep}u$. This maps restrict to an algebra   homomorphism on $\smooth(M)$. 
\end{prop}
\begin{definition}\label{cra}
 Fix a  Riemannian metric  on $M\times M$. We call a net  $T_{\ep}\in\smoothing(M)$ a local moderate approximate unit or a rapidly converging approximation if 
\begin{enumerate}
\item It is a moderate approximate unit satisfying the Definition \ref{mau}.
\item The following transfer of regularity holds:
\begin{eqnarray}\label{cap}
T_{\ep}(\maD'(M))\cap \maG^{\infty}(M)=\smooth(M).
\end{eqnarray}
\item There is no propagation of support  that is, for any $\delta>0$  there exists a decomposition of  the form \[T_{\ep}=L_{\ep}+N_{\ep},\]
such that $N_{\ep}$ is a negligible net of operators, and $L_{\ep}$ is  supported in a $\delta$  neighbourhood of the diagonal in $M\times M$.  \end{enumerate}
Thus for a local moderate approximate unit $\Op{supp}(T_{\ep}(u))=\Op{supp}(u)$.
\end{definition}
The local moderate units are also preserved under diffeomorphisms.
We shall refer to the condition \eqref{cap} on a moderate approximate unit as tameness condition. 
\begin{prop}\label{local}
If $T_{\ep}$ is local then  the  map $u\rightarrow T_{\ep}u$ is a sheaf-morphism on the sheaf $\mathscr{D}(M)\rightarrow \maG^s(M)$. In particular  local units preserve supports and  singular supports of distributions.
\end{prop}
 \begin{proof}
This is accomplished as usual  by covering  by precompact  open sets  and cut offs. Here are  the details for completeness.

Let $U\subseteq M$  be an open subset. 
Now cover $U$  by an open cover $U_{\lambda},\lambda\in\Lambda$ such that the closure $\overline{U_{\lambda}}$  is compact in $U$. and let $\phi_{\lambda}\in\smooth_c(U)$ be  such that $\phi_{\lambda}\equiv 1$ on some neighbourhood of $U_{\lambda}$. Define  the map $T_U:\maD'(U)\rightarrow\maG^s(U_{\lambda})$ by $T_{\lambda}(u):=T_{\ep}(\phi_{\lambda}u)_{|_{U_{\lambda}}}$. 
 Then we check that:
\[T_{\lambda}(w)_{|_{U_{\lambda}\bigcap U_{\mu}}}=T_{\mu}(w)_{|_{U_{\lambda}\bigcap U_{\mu}}}.\]
 This follows immediately as $(\phi_{\lambda}-\phi_{\mu})w$ is supported away from $U_{\lambda}\bigcap U_{\mu}$ implies that  $T_{\ep}((\phi_{\lambda}-\phi_{\mu})w)$ is also  supported  away from $U_{\lambda}\bigcap U_{\mu}$ by Definition \ref{cra}. Thus there exists a $T_U(w)\in\maG^s(U)$   such that  $T_U(w)_{|_{U_{\lambda}}}=T_{\lambda}(w)$. The  map $w\rightarrow T_U(w)$  by a similar argument is  independent of the covering $U_{\lambda}$ and the cut off functions $\phi_{\lambda}$ and provide the required  sheaf morphism.

For a distribution $u\in\maD'(M)$  and an  open set $U\subset M$  it follows  then by  \eqref{cap} that $T_{\ep}(u)_{|_U}$ is in $\maG^{\infty}(U)$ precisely if $u_{|_U}\in\smooth(U)$. Thus $T_{\ep}$ preserves the singular support.
 \end{proof}
We shall denote  the set of local moderate  approximate units  on $M$ by $\maU_{\Op{loc}}(M)$  or simply $\maU_{\Op{loc}}$.
 There is of course a plentiful supply of moderate approximate units. All the examples given below are of local moderate approximate units.
\begin{eg}\label{funct_calc}
 Let $\Delta$ be  the Laplace operator associated to  a Riemannian manifold $M$. Let $F\in\Sch(\RR)$ be a Schwartz function on the reals such that $F$ is identically 1 near origin. Let $F_{\ep}(x):=F(\ep x)$. Then by applying standard functional calculus $F_{\ep}(\Delta)$ is a moderate approximate unit. All the asymptotic properties follow from Weyl's estimate on  eigenvalues of $\Delta$. In addition $F_{\ep}(\Delta)$  is invariant under isometries .(See \cite{dave} for details.)
\end{eg}
As  a special case of the above example  consider a  compact Lie group $G$ and let $\hat{G}$ denote the set of all irreducible  representations  of $G$. Let $\pi\in\hat{G}$ be an irreducible representation  and let
\[\chi_{\pi}(g)=\Op{tr}[\pi(g)],\] 
be the character of $\pi$,  and is well defined as $\pi$ is necessarily  finite dimensional with dimension  denoted by $d_{\pi}$. As a consequence  of Peter Weyl theorem one  can  obtain (see \cite{Wong}),
\[f=\sum_{\pi\in\hat{G}}d_{\pi}\chi_{\pi}\underset{G}{*}f \qquad f\in L^2(G).\]
 Let $\Delta_G$ denote the Laplace operator on $G$. obtained from a  basis of the Lie algebra $\mathcal{G}$ then $\chi_{\pi}$ are eigenfunctions of $\Delta_G$ with eigenvalues $\lambda_{\pi}$. Hence if $F$  is a Schwartz function as in the above example then
\begin{align*}
F_{\ep}(\Delta_G)f&=\sum_{\pi\in\hat{G}}d_{\pi}F_{\ep}(\Delta_G)(\chi_{\pi}\underset{G}{*}f)\\
&=\sum_{\pi\in\hat{G}}d_{\pi}F_{\ep}(\Delta_G)(\chi_{\pi})\underset{G}{*}f,\\
&=\sum_{\pi\in\hat{G}}d_{\pi}F({\ep}\lambda_{\pi})(\chi_{\pi})\underset{G}{*}f
\end{align*}
here we have used  the left invariance of $F_{\ep}(\Delta_G)$.  

We now move to some noncompact examples of Lie group where convolution shall play  an important part.
\begin{eg}
Due to noncompactness of $\RR^n$ one has to modify the  above notions slightly and work with compactly supported distributions.
%%%%%%%%%%%%%%%%
\begin{definition}
A  rapidly converging approximate unit on $\RR^n$  is a net of operators $T_{\ep}$ such that their kernels  $\operatorname{ker_T}_{\ep}\in\Sch(\RR^n\times \RR^n)$.
and the following holds.
\begin{enumerate}[(a)]
\item For  any $u\in \compsupp(\RR^n)$ the net
$T_{\ep}(u)$ is a moderate net in $\smooth(\RR^n)$. and 
\[\underset{\ep\rightarrow 0}{lim}\,T_{\ep}(u)=u.\]
That is $T_{\ep}:\compsupp(\RR^n)\rightarrow\maG^s(\RR^n)$ is a well defined injective map.
\item For a compactly supported function $f\in\smooth_c(M)$  the regularization converges rapidly  that is $T_{\ep}(f)-f$ is a negligible net in special colombeau algebra $\maG^s(\RR^n)$.
\item The regularization preserves supports. that is for $u\in\compsupp(\RR^n)$ $\operatorname{supp}u=\operatorname{supp_g}T_{\ep}(u)$
\end{enumerate}
\end{definition}
   As in the original construction of Colombeau \cite{c1} a moderate approximate unit can be constructed from a mollifier $\rho\in\Sch(\RR^n)$ satisfying the following  conditions:
\begin{eqnarray}  \label{nomoment}
\int_{\RR^n}\rho(x)dx=1\quad\int_{\RR^n}x^{\alpha}\rho(x)dx=0\hskip 0.2in \alpha\in\NN_+^n.
\end{eqnarray}
Then the net of functions  $\rho_{\ep}(x):=\frac{1}{\ep^n}\rho(\frac{x}{\ep})$ is a delta net and 
  convolution with such a delta net provides an example of moderate approximate unit.  An important characteristic of these approximate units is their equivariance with respect to the Euclidian translations.
\end{eg}
Note that  on $\RR$  a mollifier satisfying \eqref{nomoment} can be obtained as the  Fourier transform of the  function $F$  used in  Example \ref{funct_calc}.
\begin{eg}
 The previous example  can be modified  in many ways.  For instance let $\rho$ be a Schwartz function on $\RR^2$ satisfying \eqref{nomoment}. Let $\tilde{\rho}_{\ep}(x,y):=\frac{1}{\ep^3}\rho(\ep^{\!-\!1} x,\ep^{\!-\!2}y)$. Then convolution with this new delta net  continues  to provide an moderate approximate unit.
\end{eg}

On $\RR^{2n+1}$ convolution in Heisenberg group with a delta net provides moderate approximate units. We provide an elementary construction in all detail  in the following subsection.

\subsection{The Heisenberg group}\label{heis}
Let $H_n:=\RR^{2n}\times \RR$ be the Heisenberg group with the usual composition:
\[(x,\xi,t)\circ(y,\eta,s)=\left(x+y,\xi+\eta,t+s+\frac{1}{2}(x.\eta-y.\xi)\right).\]

 The usual volume form $dxd\xi dt$ is invariant under both left and right translations that is:
\[L^*_pdxd\xi dt=R^*_pdxd\xi dt=dxd\xi dt\quad p\in H_n.\]
The Lie algebra $\mathfrak{h}_n$ of the Heisenberg group  is generated by  left invariant vector fields:

\begin{align*}
X_i&=\frac{\partial}{\partial x_i}-\frac{\xi_i}{2}\frac{\partial}{\partial t}\\
\Xi_i&=\frac{\partial}{\partial \xi_i}+\frac{x_i}{2}\frac{\partial}{\partial t}\\
T&=\frac{\partial}{\partial t}.
\end{align*}
 Let $\Delta$ be the  associated Laplace operator,
\[\Delta:=T^2+\sum_i X_i^2+\Xi_i^2.\]

 Then the metric associated  to  it is a left invariant metric on $H_n$. For example in case of $H_1=\RR^3$ it can be given by :
\[\mathfrak{G}=
\begin{pmatrix}
1+\frac{y^2}{2}&-\frac{xy}{2}&\frac{y}{2}\\
-\frac{xy}{2}&1+\frac{x^2}{2}&-\frac{x}{2}\\
\frac{y}{2}&-\frac{x}{2}&1
\end{pmatrix}
\]
We shall need the  operator $\Delta$  apply standard elliptic regularity argument to our constructions.

Now to construct moderate approximate units on $H_n$ we  can procede like in the Euclidian case discussed before. L
We shall need the following:
\begin{prop}Let $\rho\in \Sch(\RR^{2n+1})$ be a mollifier that  satisfies Equation \eqref{nomoment}. Let  $\rho_{\ep}(x):=\frac{1}{\ep^n}\rho(\frac{x}{\ep})$. Then $T_{\ep}(u)=\rho_{\ep} \heis u$ is a rapidly converging approximation.
\end{prop}
\begin{enumerate}[(a)]
\item 
By elliptic regularity of $\Delta$ on any bounded domain  given any compactly supported distribution $u\in\compsupp(H_n)$ there exists  a compactly supported continuous function $f$ supported in any neighbourhood of $\operatorname{supp}u$ such that for some constants $c_j$`s
\[\sum_jc_j\Delta^jf=u.\]
 Thus for any smooth $\rho$
\begin{align*}
u\heis \rho&=\sum c_j\Delta^jf\heis \rho\\
&=\sum c_j\Delta^j(f\heis\rho).
\end{align*}
Or it suffices to check the asymptotics for the case when $u=f$ is continuous. Now the estimate can be directly be obtained from the integral formula:
\[\rho_{\ep}\underset{H_n}{*} f(p)=\int_{G}\rho_{\ep}(q)f(pq^{-1})dq.\]

\item  Let $f$ be a compactly supported smooth function. then:

Let $\Op{dist^H}_{\ep}(p,q)=p-(\ep q^{-1})p$ be the euclidean difference.
then 
\[\Op{dist^H}_{\ep}(p,q)=-\ep q+\frac{\ep}{2}\omega(\bar{p},\bar{q})\]

where $\bar{(x,\xi,t)}=(x,\xi)$  and $\omega$ is the standard symplectic form on $\RR^{2n}$.
Now remembering that $\rho$ satisfies \eqref{nomoment}
\begin{align*}
f\heis\rho_{\ep}(p)-f(p)&=\int f(q^{-1}p)\rho_{\ep}(q)dq-f(p)\\
&=\int f((\ep q^{-1})p)\rho(q)dq-f(p)\\
&=\int f(p-\Op{dist^H}_{\ep}(p,q))\rho(q)dq-f(p)\\&=\int \left( f(p-\Op{dist^H}_{\ep}(p,q))-f(p)\right)\rho(q)dq.
\end{align*}
Applying Taylor expansion  we get
\begin{align*}
f\heis\rho_{\ep}(p)-f(p)&=\int \left( \sum_{|\alpha|<N}\frac{\Op{dist^H}_{\ep}(p,q))^{\alpha}}{\alpha!}\partial^{\alpha}f(p)\right)\rho(q)dq+C\ep^N\\
&\sim O(\ep^N).
\end{align*}

\item One only needs to observe that given $\delta>0$ one can decompose $\rho_{\ep}$  into $\rho_{\ep}=w_{\ep}+v_{\ep}$ with $w_{\ep}$ supported in a $\delta$ neighbourhood of the origin and $v_{\ep}$ is negligible. Thus support of $u\heis \rho_{\ep}$ is contained in every $\delta$ neighbourhood of $\operatorname{supp}u$. 
\end{enumerate}
\begin{remark}
 Given a  smooth ``delta net'' the proof of the  above preposition holds  for any Lie group. But for a general Lie group there is no "easy" choice of a  smooth net of functions suitably approximating the delta distribution at identity similar to $\rho_{\ep}$.
\end{remark}
 We gather all the facts together to obtain,

\begin{prop}\label{group}
For a  schwart function $\rho$ satisfying \eqref{nomoment} the map
\[\rho:\compsupp(H_n)\rightarrow  \maG^s(H_n)\quad \rho(u):=\rho_{\ep}\heis u,\]
extends to a  sheaf morphism $\maD'(H_n)\rightarrow \maG(H_n)$  and   is an algebra  homomorphism on $\smooth(H_n)$.
\end{prop}
\begin{eg}
For $x,\xi\in\RR^n$ let $T_x$ and $M_{\xi}$  be the operation of translation and modulation on $\RR^n$. That is 
\begin{align*}
T_x f(t)&:=f(t-x),\\
M_{\xi}f(t)&:=e^{2\pi i \xi\cdot t}f(t).
\end{align*}
We recall that the Schr\"odinger representation is the unitary representation of the Heisenberg group  on $L^2(\RR^n)$ given by
\[\pi(x,\xi,\tau):=e^{2 \pi i  \tau}e^{\pi i x\cdot  \xi} T_xM_{\xi}\,\in\,\bf{B}(L^2(\RR^n)).\]
The Fourier transform of  a Schwartz function $F\in\Sch(\RR^{2n+1})\subset L^1(H_n)$  is  a bounded operator on $L^2(M)$ given by:
\[\hat{\pi}(F):=\int_{H_n} F(\alpha)\pi(\alpha^{\!-\!1})d\alpha.\]
In fact one computes  the kernel to be
\[\operatorname{ker}\hat{\pi}(F)(x,y)=\mathcal{F}_2\mathcal{F}_3F(y-x,\frac{y+x}{2},1)\in\Sch(\RR^n\times \RR^n).\]
 and hence $\hat{\pi}(F)$ defines a smoothing operator on $L^2(\RR^n)$.

We wish to claim that the Fourier transform with respect to the Schr\"odinger representation  maps a  delta net $\rho_{\ep}$ to a rapidly converging approximate unit. 

 We shall need the following  result.
\begin{lem}\label{cheap}
Given $f\in\mathcal{C}_c^{\infty}(\RR^n)$ there exists there exists a $g\in\Sch(\RR^n)$ and $F\in\Sch(\RR^{2n+1})$ such that
\[\hat{\pi}(F)g=f.\]
\end{lem}
\begin{proof}
Let $\phi(\tau)\in\Sch(\RR)$  be such that $\int \phi(\tau)d\tau=1$ and  let $g\in\Sch(\RR^n)$ such that $\|g\|_2=1$ then set
\begin{align*}
F(x,\xi,\tau):&=e^{ 2\pi i \tau}\phi(\tau)e^{\pi  i x\cdot \xi}\int f(t)\overline{g(t+x)}e^{2\pi i\xi.t}dt.\\
&=e^{ 2\pi i \tau}\phi(\tau)e^{\pi  i x\cdot \xi}\operatorname{V}_g f(-x,-\xi).
\end{align*}
 here $\operatorname{V}_gf(x,\xi)$ is the Short Time Fourier Transform(STFT).  The inversion formula for  STFT  is given by (see \cite{Groechenig}) :
\[f(t)=\frac{1}{\langle g,g\rangle}\int\operatorname{V}_gf(x,\xi)M_{\xi}T_x g dx d\xi.\]
Then as an immediate consequence of the inversion formula  we have
\begin{align*}
\hat{\pi}(F)g(t)&=\int F(x,\xi,\tau)e^{-2\pi i \tau}e^{-\pi ix.\xi}M_{-\xi}T_{-x}g(t)dxd\xi d\tau\\
&=\int\operatorname{V}_gf(-x,-\xi)M_{-\xi}T_{-x}g(t)dx d\xi=f(t).
\end{align*}
\end{proof}
\end{eg}
\begin{prop}
Let $\rho\in\Sch(\RR^{2n+1})$ be a Schwartz function that satisfies \eqref{nomoment} and let $\rho_{\ep}(x)=\frac{1}{\ep^{2n+1}}\rho(x{\ep})$. Then
$\hat{\pi}\rho_{\ep})$ is a rapidly converging approximation on $\RR^n$.
\end{prop}
\begin{proof}
We already know that the kernel of $\hat{\pi}(\rho_{\ep})$  is  in $\Sch(|RR^n\times \RR^n)$. We shall here only show that for a compactly supported smooth function $f\in\compsupp(\RR^N)$ the approximation $\hat{\pi}(\rho_{\ep})f$ converges to $f$ rapidly. The proof of other properties is routine.  By lemma \ref{cheap} we have
\[f=\hat{\pi}(F)g\quad \exists F\in \Sch(\RR^{2n+1})\,,\, g\in\Sch(\RR^n).\] 
Therefore
\begin{align*}
{\hat{\pi}(\rho_{\ep})f-f} &=\hat{\pi}(\rho_{\ep})\hat{\pi}(F)g-\hat{\pi}(F)g,\\
&=\hat{\pi}(\rho_{\ep}\heis F-F)g.
\end{align*}
Hence the result follows from  Proposition \ref{group}.
\end{proof}
%%%%%%%%%%%%%%%%%%%%%%%%%%%%%%%%%%%%%%%%%%%%%%%%%%%%%%%%%%%%%%%%%%%%%%%%%%%%%%%%%%%%%%%%%%%%%%%%%%%%%%%%%%%%%%%%%%55
\section{Global algebras of generalized functions}\label{glob_alg}
In this section we shall construct various candidates for algebras of generalized functions. We shall generally refer to all of them  as ``full type algebras''. They shall depend on choice of a set of regularizing processes. For a particular problem such a set of regularizing process might be chosen depending on the symmetries involved. We shall provide a few toy examples in Section \ref{examples}
  On closed manifold $M$  both $\smooth(M)$ and $\smoothing(M)$  forms  nuclear Frech\'et space with jointly continuous  multiplication. 
 Let $E^{\infty}$  be the set of  all maps 
\[\phi:\smoothing(M)\rightarrow \smooth(M)\quad \phi\,\textrm{ a Frech\'et  smooth map}.\]
 Being a space of maps into a commutative algebra $E^{\infty}$  is an algebra under pointwise  operations. 

 By evaluation on operators a distribution defines a map 
\begin{eqnarray}
\label{eval}\Theta_u:\Psi^{\!-\!\infty}(M)\rightarrow \smooth(M)\,\,\,
 \Theta_u(T):=T(u)\,\,\forall \, T\in\smoothing(M).
\end{eqnarray}
 If $K_T(x,y)$ denotes the  integral kernel of $T$ then  the evaluation on $u\in\maD'(M)$ is given by
\[\Theta_u(T)=T(u)=\ip{u(y),k_T(x,y)}.\]
 Thus $u\rightarrow \Theta_u$ is a map from $\maD'(M)$ into $\mathfrak{L}(\smoothing(M),\smooth(M))$ the space of continuous linear maps between these two Frech\`et spaces. This map is clearly an injective map as a section $f\in\smooth(M,\Omega)$ can be used to define a  smoothing map $u\rightarrow u(f)$ which separates distributions. 

 The map $u\rightarrow \Theta_u$ extends to a map on the tensor algebra $\mathfrak{T}\maD'(M)\rightarrow E^{\infty}$ by
\[\rho(u_o\otimes u_1\otimes \ldots\otimes u_r)(T)=T(u_0)T(u_1)\ldots T(u_r).\]
 
   The restriction   to smooth functions  $f\rightarrow \Theta_f\,f\in\smooth(M)$ is however not an algebra homomorphism. By choosing  a subalgebra of $E^{\infty}$  and quotienting  by an ideal containing $f-\Theta_f$ for all $f\in\smooth(M)$  one can easily obtain  an algebra homomorphism on $\smooth(M)$.  There is a choice of  such   a morphism for   any  collection of regularizing processes.

 Let $\maL\subset \maU$ be a set of moderate approximate units or regularizing processes (see Definition \ref{mau}). We say that a smooth map $\phi:\smoothing(M)\rightarrow \smooth(M)$ is moderate over $\maL$ if for all $T_{\ep}\in\maL$ the evaluation  $\phi(T_{\ep})\in E_{\smooth(M)}$  is a   moderate net. Or more elaborately  $\phi$ is moderate over $\maL$ if given an approximate unit $T_{\ep}\in \maL$ and a continuous  seminorm $\rho$ on $\smooth(M)$ there exists an integer $N$ such that:
\[\rho(\phi(T_{\ep}))\sim O(\ep^N).\]
   
 The set of all moderate maps over $\maL$ shall be denoted by $E_{\maL}(M)$.

Similarly $\phi\in E^{\infty}(M)$ is said to be negligible over $\maL$ if for any approximate unit $T_{\ep}$ in $\maL$  $\phi(T_{\ep})$ is a negligible net of smooth functions.  That is for all any seminorm $\rho$  on $\smooth(M)$,
\[\rho(\phi(T_{\ep}))\sim O(\ep^N)\quad \textrm{for all}\, N\in\ZZ.\]

The set of all negligible maps over $\maL$ shall be denoted by $N_{\maL}(M)$. Of course it suffices to check the estimates for $E_{\maL}(M)$ and $N_{\maL}(X)$ only for a family of seminorm that generate the locally convex topology on $\smooth(M)$.

 One can readily check that $N_{\maL}(M)$ is an ideal in $E_{\maL}(M)$. One way to see this is that given a differential operator $D$ there exists differential operators $P_i,Q_i$ such that for any two smooth functions $g,h\in\smooth(M)$
\[D(hg)=\sum P_i(h)Q_i(g).\]
Thus for instance
\[\|D(hg)\|_{L^2(M)}\leq \sum_i \|P_i(h)\|_{L^{\infty}(M)}\|Q_i(g)\|_{L^2(M)}.\]
Let $D$ be an invertible elliptic operator. Let $\phi\in E_{\maL}(M)$ and $\psi\in N_{\maL}(M)$ then for any $T_{\ep}\in\maL$
\[\|D(\phi.\psi(T_{\ep})) \|_{L^2(M)}\leq \sum_i \|P_i(\phi(T_{\ep}))\|_{L^{\infty}(M)}\|Q_i(\psi(T_{\ep}))\|_{L^2(M)}.\]
which proves that the product is in $N_{\maL}(M)$
\begin{definition} Let $\maL$ be a set of  moderate approximate units. The full algebra of generalized functions  over $\maL$ is defined as
\[\maG_{\maL}(M):=\frac{E_{\maL}(M)}{N_{\maL}(M)}.\]
\end{definition}

We note that  if $\maM\hookrightarrow \maL$  is a subset of moderate approximate units then $\maG_{\maL}(M)\hookrightarrow \maG_{\maM}(M)$ thus provides a contravariant functor from subsets of  regularizing processes to generalized functions.

Of course any distribution defines a  map  from $\smoothing(M)$ to $\smooth(M)0$ by evaluation \eqref{eval}. 
\begin{lem}
For any distribution  $u$ the map $\Theta_u\in E(\maU)$. Hence $\Theta_u\in E_{\maL}(M) $ for any set of moderate approximate units $\maL$.
Thus we have an embedding   of $\maD'(M)\rightarrow \maG_{\maL}(M)$.  This  embedding  restricts on $\smooth(M)$ to an  algebra homomorphism 
\end{lem}
\begin{proof}
 Since any $T_{\ep}\in\maU$ is by definition a moderate net of smoothing operators and a distribution $u:\smoothing(M)\rightarrow \smooth(M)$ is a continuous map,   hence  by  Lemma \ref{lcs} $\Theta_u(T_{\ep})\in E_{\smooth(M)}$..
This  embedding  restricts on $\smooth(M)$ to an  algebra homomorphism because for any $T_{\ep}\in\maU$  and any $f\in\smooth(M)$ 
\[\Theta_f(T_{\ep})-f=T_{\ep}f-f\in N_{\smooth(M)},\]
again by Definition \ref{mau}.
\end{proof}

%\section{Operators extenndeble to $\maG$ and invariant operators}

\subsection{Operatons of generalized functions }\label{operations}
Next we move to action of diffeomorphisms on these algebras and we check that the embedding of  distributions is equivariant with respect to diffeomorphism action.

\subsubsection{Action of diffeomorphism}
Let $\chi:M\rightarrow M$ be a diffeomorphism on $M$. Then $\chi$ acts on  $\smoothing(M)$ by push forward of operators as:
\[\chi_*(T)(f):=\chi^*T(\chi^{-1*}f).\]
 Let $\mu_y\Omega$  be a nonzero section of the density bundle on $M$. Let $T$  be given by  a kernel $k_T(x,y)\mu_y\in\Gamma(M\times M,\pi_2^*\Omega)$ then the kernel of $\chi_*T$ is given by  $k_T(\chi x,\chi y)\chi^{-1^*}\mu_y$.

  We extend the action of $\Op{Diff}(M)$ on $\phi\in E^{\infty}$ by 
\[{\chi^*\phi(T):=\chi^*(\phi(\chi_*(T)))}.\]
The composition can be seen as the following diagram,
\begin{center}
$ \xymatrix{ {\smoothing(M)} \ar[d]^{\chi_*} \ar@{.>}[r]^{\chi^*\phi}&{\smooth(M)}\\
{\smoothing(M)} \ar[r]^{\phi}&{\smooth(M)}\ar[u]^{\chi^*}}$
\end{center}
\begin{lem}
The embedding of $\maD'(M)$ in $E^{\infty}$ is equivariant under diffeomorphisms that is
\[\chi^*(\Theta_u)=\Theta_{\chi^*u}.\]
\end{lem}
\begin{proof}
Let  $T $  be a smoothing operator with kernel $k_T(x,y)\mu_y$ then  
\begin{align*}
\Theta_{\chi^*u}(T)&=\chi^*u(T)=\ip{u(\chi^*u(y),k_t(x,y)\mu_y}\\
&=\ip{u(y),k_T(x,\chi(y))\chi^{-1^*}\mu_y}\\
&=\chi^{-1^*}\ip{u(y),k_T(\chi(x),\chi(y))\chi^{-1^*}\mu_y}=\chi_*\Theta_u(T).
\end{align*}
\end{proof}
\begin{cor}\label{inv}
Let $X\subset \maU$ be a set of moderate units. Now let $\chi$ be a diffeomorphism then $\chi^*:\maG_X(M)\rightarrow \maG_{\chi_*(X)}(M)$  and hence if $\chi_*(X)=X$ then the action of $\chi$ descends  naturally to an action  on $\maG_X(M)$.
\end{cor}

\subsubsection{Pseudodifferential operators}
The smoothing operators $\smoothing(M)$ form an ideal in the algebra of pseudodifferential operators $\Psi^{\infty}(M)$ hence it is very easy to define an action of  pseudodifferential operator $P$ on the  space $E^{\infty}$ that extends their action on $\maD'(M)$. We define the operator on $\phi\in E^{\infty}$ by,
\begin{eqnarray}\label{pseudo}
P\phi(T):=\phi(TP),\quad \phi\in E^{\infty}\,,\, T\in\smoothing(M).
\end{eqnarray}
 This is indeed an extension of the operators on $\maD'(M)$ as 
\[P\Theta_u(T)=\Theta_u(TP)=TP(u)=\Theta_{Pu}(T).\]

Let $\hat{E}$ be the ring of  polynomially bounded smooth maps on the complex plane.  More precisely the set of all maps
\[u:\CC \rightarrow \CC \quad\exists p\in\CC[z]\,\,|\,\,|u(z)|\leq |p(z)|\forall \,\, z.\]
 Every element $u\in\hat{E}$ induces a map $u_*:\tilde{\CC}\rightarrow\tilde{\CC}$. Let $\hat{N}$ be all elements in $\hat{E}$ such that  the induce map is $0$  map or $\hat{N}=\Op{ker}(u\rightarrow u_*)$. We define 
\[ \hat{\CC}:=\frac{\hat{E}}{\hat{N}}.\]

Then it is clear that  $\hat{\CC}$ is a ring  and  all the algebras $\maG_{\maL}(M)$ are  algebras over the ring $\hat{\CC}$. by the  action 
\[u.\phi:\smoothing(M)\rightarrow \smooth(M)\qquad u.\phi(T)=u(\Op{Tr}(T))\phi(T).\]

Here  $\Op{Tr}(T)=\int_Mk_T(y,y)\mu_y$ is the operator trace of $T$.
%%%%%%%%%%%%%%%%%%%%%%%%%%%%%%%%%%%%%%%%%%%%%%%%%%%%%%%%%%%%%%%%%%%%%%%%%%%%%%%%%%%%%%%%%%%%%%
\section{Examples}\label{examples}
\subsection{Riemannian manifolds}
Let $M$ be a closed Riemannian manifold and let $\Delta$ be the associated scalar Laplace operator.Let $f\in\Sch(\RR)$ be a Schwartz function with $f\equiv 1$ near the origin . Then $T_{\ep}:=f_{\ep}(\Delta)$  is a n moderate approximate unit and let $X_f=\{T_{\ep}\}$ be the singleton set. Since the Laplace operator is invariant under  isometries $X_f$ is  invariant under the group of  isometries $\Gamma:=\operatorname{Iso}(M)$. Therefore $\maG_{X_f}(M)$ has a natural action of $\Gamma$ which is equivariant under the embedding of the distributions in  $\maG_{X_f}(M)$. 
\begin{prop}
Let $f\in\Sch(\RR) $ be a real valued Schwartz function  with  $f\equiv 1$ near the origin and $f(x)$ is monotone on $(0,\infty)$and  strictly monotonically decreasing on $(1,\infty)$.
Let $\maG^s(M)$ denote the special Colombeau algebra with smooth nets then 
$\maG_{X_f}(M)$ is  naturally algebra isomorphic  to $\maG^s(M)$.
\end{prop}
\begin{proof}
Let $\phi:\smoothing(M)\rightarrow \smooth(M)$ represent $\alpha\in\maG_{X_f}(M)$  The map
\[\rho(\alpha):=[\phi(T_{\ep})]\in\maG^s(M),\]
shall be the required isomorphism.

It  is evident that $\rho:\maG_{X_f}(M)\rightarrow \maG^s(M)$ is well defined, injective algebra morphism. To check the surjectivity of $\rho$ we first note that by our assumption on $f$ it  follows that $\ep\rightarrow T_{\ep}$ is a smooth embedded curve in $\smoothing(M)$. To see this we check,
\begin{itemize}
\item  The map $\ep\rightarrow T_{\ep}$ is injective: If $\ep\neq\delta$ then we have $f_{\ep}(x)\neq f_{\delta}(x)$ for all $x$ large enough  and hence $f_{\ep}(\Delta)\neq f_{\delta}(\Delta)$.
\item  It is an immersion as $\frac{d}{d\ep} T_{\ep}=F'_{\ep}(\Delta)\neq0$.
\item  It homeomorphism on its image: Let $\lambda_n,\phi_n$ be respectively the eigenvalues and eigenvectors of the Laplace (say ordered and counted with multiplicity.)  For any  $n\in\NN$ the  map 
\[j_n:\smoothing(M)\rightarrow \CC\quad j_n(T):=\ip{T(\phi_n),\phi_n},\]
 is a continuous map such that $j_n(F_{\ep}(\Delta))=F(\ep\lambda_n).$
At each $\delta\in(0,1)$ there exists $k$ such that $F(\delta \lambda_k)=C_{\delta}\leq \frac{1}{2}$. By monotonicity of $F$ there is an interval $(\delta-\tau,\delta+\                                tau)$ such that  The inverse image of an open ball, of radious $r<\frac{C_{\delta}}{3}$ around $C_{\delta}$ in $\CC$ under $j_k$ intersects $T_{\ep}$ in  an interval. Since $\smoothing(M)$ is a meterizable this is enough to proof that the image is  homeomorphic to $(0,1)$.
\end{itemize}

Now given an element $\beta\in\maG^s(M)$ we pick a smooth representative $u_{\ep}$ . The map $\tilde{\phi}(T_{\ep})=u_{\ep}$ is a smooth map from an embedded submanifold and hence can be extended to all of $\smoothing(M)$. This is because $\smoothing(M)$ is a nuclear Frech\'et space and hence is $\smooth$-paracompact (see \cite{KM}  Theorem 16.10). We call one such extension $\phi$ Then $\rho([\phi])=\beta$.
\end{proof}

\subsection{Symplectic manifolds}
Let $(M,\omega)$ be a closed symplectic manifold.  Recall that an almost complex structure on $M$  is   is a bundle map $J:TM\rightarrow TM$ such that $ J^2=-\operatorname{Id}$.

A metric $\mathscr{G}$  is compatible with $\omega$ if there exists an  almost complex structure $J$ such that 
\[\omega(Jx,JY)=\mathscr{G}(X,Y), \quad X,Y\,\,\textrm{vectorfields.}\]

By polar decomposition on any metric one knows that compatible metrics and almost complex structures always exist. Let $\mathscr{G}(\omega)$ be the set  of all compatible Riemannian metrics and let $\Delta(\omega)$ be the set of all  scalar Laplaces associated with compatible Riemannian metrics. Let $f\in\Sch(\RR)$ be a Schwartz function $f\equiv 1$ near origin. Let
\[X(\omega,f):=\{f_{\ep}(\Delta)|\,\Delta\in\Delta(\omega)\}.\]
 Then $X(\omega,f)$  is a set of moderate approximate units.

\begin{lem}
The set of approximate units $X(\omega,f)$  is invariant under the group of symplectic diffeomorphism $Symp(M)$.
\end{lem}
\begin{proof}
  Let $\Delta $ be  the scalar Laplace operator associated with the metric  $\mathscr{G}$ and let  $\phi $ be a  diffeomorphism then the push foreword operator $\phi_*(\Delta)$ is the scalar Laplace associated to the metric $\phi^{-1^*}(\mathscr{G})$.

Also  for  any  $f\in\Sch(\RR)$  a Schwartz function
\[f(\phi_*(\Delta))=\phi_*(f(\Delta)).\]
Thus  the desired result follows  from the fact that if  $\phi$ is a symplectic diffeomorphism then $\phi$ preserves the  space of  compatible metrics $\mathscr{G}(\omega)$.
\end{proof}
    Thus combining  with Corollary \ref{inv} we have the following theorem.
\begin{theo}
The full type algebra defined by the set of approximate units $X(\omega,f)$ has an action of $Symp(M)$ and the natural embedding of distributions is equivariant under this action.
\end{theo}

%%%%%%%%%%%%%%%%%%%%%%%%%%%%%%%%%%%%%%%%%%%%%%%%%%%%%%%%%%%%%%%

\section{Regularity}\label{PT}
In this section we study the  regularity structure of  distributions  within the settings of our generalized functions.

We first  consider regularity of maps between  Frech\'et spaces.  Recall that a  grading on a Frech\'et space $X$  is a  sequence of  seminorms $\|\,\|_n$ that  is increasing (that is  $\|\,\|_1\leq\|\,\|_2\leq\ldots $) and  generates the locally convex topology  on $X$. We refer to \cite {Hamilton} for further  study. 
\begin{definition} 
Let $X$ and $Y$  be graded Frech\'et spaces.  We denote by $\|.\|_n$ and $\|.\|'_n$ the $n$-th graded norm on $X$  and $Y$ respectively. We say that a Frech\'et  smooth map $\phi:X\rightarrow Y$ is polynomially  tame if there exist $b,k\in\NN$  and some $r\in\ZZ$ such that
\begin{eqnarray}\label{tamer}
\|\phi(x)\|'_n\leq C\|x\|^k_{n+r}\quad \textrm {for all}\, n\geq b+|r|.
\end{eqnarray}
Here $C>0$ is a constant that depends only on $n$.The number $r$  is called the degree of tameness and the set of all maps  of tameness degree $r$  is denoted by $\Op{PT}^r(X,Y)$. Let $\Op{PT}(X,Y):=\cup_r\Op{PT}^r(X,Y)$.

A polynomially tame map is called regular if there exists a $k$  and there exists  a $b$ such that for any degree of  tameness $r\in\ZZ$ the condition \eqref{tamer} holds. It is clear that
 \[\Op{Reg}(X,Y)\subseteq \cap_r\Op{PT}^r(X,Y).\]
\end{definition}
 We would also require the following tameness  property for  a associative multiplication on a Frech\'et  space.
\begin{definition}\label{FA}
We say that $X$ is a Frech\'et  algebra if it is a Frech\'et space with an associative product and the multiplication is jointly continuous. A graded Frech\'et algebra is a Frech\'et algebra  such that the   multiplication satisfies a tameness  condition namely there exist $b,r_1,r_2\in\NN$ such that
\[\|x\cdot y\|_n\leq C\|x\|_{n+r_1}\|y\|_{n+r_2}\quad \forall \,\,n\geq b.\]
\end{definition}
With the above definition the following   lemma is self -evident.
\begin{lem}
Let $Y$ be a graded Frech\'et algebra. And let $X$ be any graded Frech\'et space. then the space of polynomially tame maps from $X$ to $Y$ is an algebra under pointwise operations. The regular maps form a subalgebra.
\end{lem}
The  algebra of polynomially tame  maps  $\operatorname{PT}(X,Y)$ and the regular maps as $\operatorname{Reg}(X,Y)$  depend on  not just the  topology of $X$ and $Y$  but also on the choice of the grading structures   on them. Equivalent gradings provide the same algebras. We shall grade  $\smooth(M)$  with Sobolev norms. Let $\Delta$ denote  the Laplace operator on $M$ associated to a metric then,
\[\|f\|_n:=\|(1+\Delta)^{\frac{n}{2}}f\|_{L^2(M)}.\]
 With the above grading  $\smooth(M)$  is a graded Frech\'et algebra  in the sense of Definition  \ref{FA}.

With the metric we identify $\smoothing(M)=\smooth(M\times M)$.
Define a grading on $\smoothing(M)$  by:
\[
\|T\|_n:=\sum_{q+p=n,q,p\geq -n}\|(1+\Delta)^{\frac{p}{2}}T(1+\Delta)^{\frac{q}{2}}\|_{\operatorname{HS}}.
\]
For any operator $D:L^2(M)\rightarrow L^2(M)$ we denote by $\|D\|_{\operatorname{HS}}$  its  Hilbert-Schmidt norm.
 With the above choice of gradings  we  shall denote $\operatorname{PT}(M):=\operatorname{PT}(\smoothing(M),\smooth(M))$ and $\operatorname{Reg}(M):=\operatorname{Reg}(\smoothing(M),\smooth(M))$.
\begin{prop}\label{regularity} Let $\smoothing(M)$ and $\smooth(M)$ be graded as above then:
\begin{enumerate}
\item  A smooth map $\phi:\smoothing(M)\rightarrow \smooth(M) $ is a regular map then $\phi(T_{\ep})\in\maG^{\infty}(M)$ for any  moderate approximate unit $T_{\ep}$.
\item All polynomially tame  maps $\phi\in\Op{PT}(M)$  belong to $E(\maU)$.
 \item For any distribution $u\in \maD'(M)$ the image $\Theta_u:\smoothing(M)\rightarrow \smooth(M)$ is a polynomially tame map. In fact
\[u\in H^k(M)\leftrightarrow \Theta_u\in\Op{PT}^{-k}(M).\]
\end{enumerate}
\end{prop}
\begin{proof}We recall that $\maG^{\infty}(M)$ is a subalgebra of  the  the special algebra $\maG^s(M)$   such that an  element $x$ is moderate of same order  with respect to all seminorms on $\smooth(M)$.
\begin{enumerate}
\item By assumption  $T_{\ep}$ is a moderate net of smoothing operator hence for any $b$ there exists an $M$ such that  $\|T_{\ep}\|_b^k\sim O(\ep^M)$. Since $\phi\in\operatorname{Reg}(M)$ by definition there exist $k,b$  satisfying \eqref{tamer} for any $r$.In particular set $r=b-n$ for $n$  large enough we have
\[\|\phi(T_{\ep})\|_n\leq C\|T_{\ep}\|_b^k\sim O(\ep^M).\]
 Thus the net $\phi(T_{\ep})$  is in $\maG^{\infty}(M)$.
\item  Again follows from moderateness of $T_{\ep}\in \maU$.
\item  Let $u\in H^k(M)$  be a distribution then:
\begin{align*}
\|\Theta_u(T)\|_n&=\|T(u)\|_n=\|(1+\Delta)^{\frac{n}{2}}T(u)\|_{L^2(M)}\\
&=\|(1+\Delta)^{\frac{n}{2}}T(1+\Delta)^{-\frac{k}{2}}\left( (1+\Delta)^{\frac{k}{2}}(u)\right)\|_{L^2(M)}\\
&\leq \|(1+\Delta)^{\frac{n}{2}}T(1+\Delta)^{-\frac{k}{2}}\|_{HS}\|(1+\Delta)^{\frac{k}{2}}(u)\|_{L^2(M)}\\
&=C\|T\|_{n-k}.
      \end{align*}
\end{enumerate}
\end{proof}
\begin{cor}\label{MO}
The only distributions which  give rise to regular maps are smooth functions. That is
\[\maD'(M)\cap \operatorname{Reg}(M)=\smooth(M).\]
\end{cor}
\begin{proof}
Let  $T_{\ep}=F_{\ep}(\Delta)$  then we know that (see\cite{dave}) 
\[F_{\ep}(\Delta)(\maD'(M))\cap \maG^{\infty}(M)=\smooth(M).\]
Hence by part (1) of Proposition  \ref{regularity} if  $\Theta_u$ is  in $\operatorname{Reg}(M)$ then $u$ must be smooth.
Also from the proof of part  (2) it is obvious that any smooth function  defines a map in $\operatorname{Reg}(M)$.
\end{proof}
In view of the above result we regard the  subalgebra  $\operatorname{Reg}(M)$  as an analogue of Oberguggenberger's algebra $\maG^{\infty}(M)$.
It provides some regularity features  to the  space  $\maG_{\maL}(M)$ consistant with the regularity  of distributions. It is clear that $\operatorname{PT}(M)$ and $\operatorname{Reg}(M)$  are  modules over $\smooth(M)$  with the  natural  action as
\[f.\psi(T):=\Theta_f\cdot \psi(T)\quad \forall f\in\smooth(M)\,\,\psi\in PT(M).\]
Note  that this module action restricted to $\maD'(M)\hookrightarrow \Op{PT}(M)$  is not the usual  module action  of $\smooth(M)$ on distributions. Let $M_f$ denote  multiplication  by a smooth function $f$ on space of  distributions $\maD'(M)$. Then  we define
\begin{eqnarray}\label{module}
M_f.\phi(T):=\phi(T.M_f)\quad \phi\in E(\maU)\, T\in\smoothing(M).
\end{eqnarray}
\begin{definition}
The support of $\phi\in E(\maU_{\Op{loc}})$   is the  complement of the  biggest open set $U\subseteq M$ such that $f.\phi\in N(\maU_{\Op{loc}})$ for any function $f$ supported in $U$.
 The singular  support of $\phi\in E(\maU)$   is the  complement of the  biggest open set $U\subseteq M$ such that $M_f.\phi\in \operatorname{Reg}(M)$ for any function $f$ supported in $U$.
\end{definition}
\begin{lem}
For a distribution $u\in \maD'(M)$ and any local moderate approximate unit $T_{\ep}$, the following are true.
\begin{enumerate}[(a)]
\item $\operatorname{supp}u=\operatorname{supp}\Theta_u=\operatorname{supp}T_{\ep}u.$
\item $\operatorname{singsupp}u=\operatorname{singsupp}\Theta_u=\operatorname{singsupp}T_{\ep}u.$
\end{enumerate}
\end{lem}
\begin{proof}
The statements $ \operatorname{supp}u=\operatorname{supp}T_{\ep}$ and $ \operatorname{singsupp}u=\operatorname{singsupp}T_{\ep}u$ follow from Lemma \ref{local}. By definition $\operatorname{supp}\Theta_u=\cap_{T_{\ep}\in\maU_{loc}}\operatorname{supp}T_{\ep}u$. 

Also it is immediate from the definition that $M_f\Theta_u=\Theta_{fu}$ Hence the  equality  of  singularsupport follows from Corollary \ref{MO}
\end{proof}

Now if $P$ is a pseudodifferential operator of order $m$ then we study the action of $P$  on polynomially tame  maps. First a quick observation that right multiplication  by $P$  on $\smoothing(M)$  namely the map  $T\rightarrow TP$  is tame of tameness $m$. To see this one notes that the operator $(1+\Delta)^{\frac{m}{2}}$ generates   $\Psi^m(M)$  as a  left-module (and also right-module) over $\Psi^0(M)$. Thus . Therefore we set set
\[P=P_0(1+\Delta)^{\frac{m}{2}}\quad P_0 \in\Psi^0(M).\]
 By the same token  for $k$ an integer multiple of $\frac{1}{2}$  we find an order $0$ operator $T_k$ such that
\[[P_0,(1+\Delta)^k]=(1+\Delta)^{k-1}T_k.\]
Puttine this together we have
\begin{align*}
\|TP\|_{p,q}&=\|(1+\Delta)^{\frac{q}{2}}TP(1+\Delta)^{\frac{p}{2}}\|_{HS}=\|(1+\Delta)^{\frac{q}{2}}TP_0(1+\Delta)^{\frac{p+m}{2}}\|_{HS}\\
&\leq \|(1+\Delta)^{\frac{q}{2}}T(1+\Delta)^{\frac{p+m}{2}}P_0\|_{HS}+\\
&\qquad\qquad\|(1+\Delta)^{\frac{q}{2}}T[P_0,(1+\Delta)^{\frac{p+m}{2}]}\|_{HS}\\
&\leq \|P_0\|\|(1+\Delta)^{\frac{q}{2}}T(1+\Delta)^{\frac{p+m}{2}}\|_{HS}+\\
&\qquad\qquad\|T_{\frac{p+m}{2}}\|\|(1+\Delta)^{\frac{q}{2}}T(1+\Delta)^{\frac{p+m-1}{2}}\|_{HS}\\
&\leq C\|T\|_{p+q+m}
\end{align*}
In particular this  implies that
\[\|TP\|_n\leq C\|T\|_{n+m}.\]

If $\phi\in\Op{PT}^r(X,Y)$  and $\tau\in\Op{PT}^m(Y,Z)$ then it is obvious that $\tau\circ\phi\in\\Op{PT}^{r+m}(X,z)$  Hence we have the following result.
\begin{prop}
Let $P$ be an order $m$ pseudodifferential operator than  it extends to a map $P\Op{PT}^r(M)\rightarrow \Op{PT}^{r+m}(MM)$  by the map $P(\phi)(T):=\phi(TP)$ as defined in \eqref{pseudo}.
\end{prop}
Now in  exact analogy with classical notion of wavefront set of a distribution  to the wavefront set of a generalized function as follows. Recall that for a pseudodifferential operator $P$  let $\sigma(P)$ denote the principal symbol of $P$.  Then  we denote as usual the characteristic set of $P$ by $\Op{Char}(P)=\sigma^{-1}\{0\}\subseteq T^*M$. We denote  by $\Psi^M_{cl}(M)$ all classical pseudodifferential operator of order $M$.
\begin{definition}
The wavefront set of a generalized function $\phi\in E(\maU)$  would be given by:
\[WF_R(\phi):=\bigcap_{P\phi\in\Op{Reg}(M)} \Op{Char}(P),\quad P \in\Psi^0_{cl}(M).\]
\end{definition}
\begin{prop}
Let $u\in \maD'(M)$  be a distribution then
\[WF_R(\Theta_u)=WF \, u.\]
\end{prop}
\begin{proof}
Since $P\Theta_u=\Theta_{Pu}$ it follows from Corollary \ref{MO} that  
\[P\Theta_u\in\operatorname{Reg}\longleftrightarrow\Theta_{Pu}\in \operatorname{Reg}(M)\longleftrightarrow Pu\in\smooth(M).
\]
Hence
\begin{align*}
WF_R(\Theta_u)&=\bigcap_{P\Theta_u\in\Op{Reg}(M)} \Op{Char}(P)\\
&=\bigcap_{Pu\in\smooth(M)} \Op{Char}(P)=WF(u).
\end{align*}
\end{proof}

%%%%%%%%%%%%%%%%%%%%%%%%%%%%%%%%%%%%%%%%%%%%%%%%%%%%%%%%%%%%%%%%%%%%%%%%%%%%%%%%%%%%%%%%%%%%%%%%%%%%%%%%%%%%%%%%%%%%%%%%%%%%%%%%%%%%%%%%%%%%%%%%%%%%%%%%%%%%%%%%%%%%%%%%%%%%%%%%%%%%%%%%%%%%%%%%%%%%%%%%%%%%%%%%%%%%%%%%%%%%%%%%%%%%%%%%%%%%%%%%%%%%%%%%%%%%%%%%%\\

%%%%%%%%%%%%%%%%%%%%%%%%%%%%%%%%%%%%%%%%%%%%%%%%%%%%%%%%%%%%%%%%%%%%%%%%%%%%%%%%%%%%\\

%%%%%%%%%%%%%%%%%%%%%%%%%%%%%%%%%%%%%%%%%%%%%%%%%%%%%%%%%%%%%%%%%%%%%%%%%%%%%%%%%%%%%%%%%%

%%%%%%%%%%%%%%%%%%%%%%%%%%%%%%%%%%%%%%%%%%%%%%%%%%%%%%%%%%%%%%%%%%%%%%%%%%%%%

\end{document}